\documentclass{amsart}

\usepackage{graphicx}
\usepackage{amsmath}
\usepackage{amscd}
\usepackage{amsfonts}
\usepackage{amssymb}
\usepackage{color}
\makeatletter
\def\@cite#1#2{{\m@th\upshape\bfseries%
[{#1\if@tempswa{\m@th\upshape\mdseries, #2}\fi}]}}
\makeatother

\numberwithin{equation}{section}

\newtheorem{theorem}{Theorem}[section]
\newtheorem{lemma}[theorem]{Lemma}
\newtheorem{proposition}[theorem]{Proposition}

\newtheorem{corollary}[theorem]{Corollary}

\theoremstyle{definition}
\newtheorem{definition}[theorem]{Definition}
\newtheorem{example}[theorem]{Example}

\newcommand{\be}{\begin{equation}}
\newcommand{\ee}{\end{equation}}
\newcommand{\bes}{\begin{equation*}}
\newcommand{\ees}{\end{equation*}}

\newcommand{\cH}{\mathcal{H}}
\newcommand{\cK}{\mathcal{K}}

\newcommand{\cA}{\mathcal{A}}

\newcommand{\cO}{\mathcal{O}}

\newcommand{\cT}{\mathcal{T}}

\newcommand{\Aut}{\operatorname{Aut}}

\newcommand{\id}{\operatorname{id}}

\renewcommand{\phi}{\varphi}

\newcommand{\proj}{\operatorname{proj}}
\newcommand{\Out}{\operatorname{Out}}
\newcommand{\Inn}{\operatorname{Inn}}

\begin{document}

\title{Automorphisms and dilation theory of triangular UHF algebras}
\author{Christopher Ramsey}
\address{University of Waterloo, Waterloo, ON, Canada}
\email{ciramsey@uwaterloo.ca}

\begin{abstract}
We study the triangular subalgebras of UHF algebras which provide new examples of algebras with the Dirichlet property and the Ando property. This in turn allows us to describe the semicrossed product by an isometric automorphism. We also study the isometric automorphism group of these algebras and prove that it decomposes into the semidirect product of an abelian group by a torsion free group. Various other structure results are proven as well.
\end{abstract}

% for regular amsart article

%\subjclass{47L40, 47L55, 47A20}
% for birkjour class article


\maketitle

\section{Introduction}
A unital non-selfadjoint operator algebra is a triangular UHF algebra if it is the closed union of a chain of unital subalgebras each isomorphic to a full upper triangular matrix algebra. That is, such an algebra can be thought of as the upper triangular part of a UHF algebra. These were extensively studied by Power \cite{Power} and many others in the early 90's.  

In their recent paper \cite{DavKat}, Davidson and Katsoulis refine various notions of dilation theory, commutant lifting and Ando's theorem for non-selfadjoint operator algebras and show that these notions become simpler when the algebras have the semi-Dirichlet property. Moreover, if the operator algebra has this nice dilation theory then one can describe the C$^*$-envelope of the semicrossed product of the operator algebra by an isometric automorphism. However, almost all examples of such algebras arose from tensor algebras of C$^*$-correspondences, the exception being given recently by E. T. A. Kakariadis in \cite{Kakariadis13}, which leads to the question whether other examples exist. While it is unknown (at least to the author) whether a triangular UHF algebra is isomorphic to some tensor algebra it does provide a new example of an operator algebra which has the Dirichlet property and the Ando property. 

This paper also addresses the isometric automorphism group of such triangular UHF algebras. We prove in section 3 that this group can be decomposed into a semidirect product of approximately inner automorphisms by outer automorphisms and that the outer automorphism group is torsion free. Section 4 provides a different proof to that of Power's in \cite{Power2} showing that the outer automorphism group of the triangular UHF algebra with alternating embeddings is determined by a pair of supernatural numbers associated to the algebra. Section 5 develops a method of tensoring the embeddings of two triangular UHF algebras to create a new algebra which combines the automorphic structure of both, giving a slightly richer perspective on what groups one can obtain.

%%%%%%%%%%%%%%%%%%%%%%%%%%%%%%%%%%%%%%%%
%%%%%%%%%%%%%%%%%%%%%%%%%%%%%%%%%%%%%%%%

\section{Triangular UHF algebras}

A C$^*$-algebra is called {\em uniformly hyperfinite} ({\em UHF}) (or a {\em Glimm algebra}) if it is the closed union of a chain of unital subalgebras each isomorphic to a full matrix algebra. In other words, suppose we have integers $k_n, n\in \mathbb N$ such that $k_n | k_{n+1}$, for all $n$, and unital C$^*$-algebra embeddings $\varphi_n: M_{k_n} \rightarrow M_{k_{n+1}}$ then $\mathfrak A_\varphi = \overline{\bigcup_{n} M_{k_n}}$ is a UHF algebra.
Such a sequence of integers $k_n | k_{n+1}$ defines a formal product $\delta(\mathfrak A_\varphi) := \prod_{p \ {\rm prime}} p^{\delta_p}$, where $\delta_p \in \mathbb N \cup \{\infty\}$, called a {\em supernatural number} or {\em generalized integer}.
A famous theorem of Glimm's \cite{Glimm} states that two UHF algebras are isomorphic if and only if they have the same generalized integers. In particular, the choice of unital embeddings does not make a difference. See \cite{Dav, Power} for more on UHF algebras and approximately finite-dimensional (AF) C$^*$-algebras, where such an algebra is defined to be a closed union of a chain of finite dimensional subalgebras.

Let $\cT_k$ be the upper triangular matrices of $M_k$ then we have the following definition:
\begin{definition}
Consider a UHF algebra $\mathfrak A_\varphi = \overline{\bigcup_{n} M_{k_n}}$ where $\varphi_n: M_{k_n} \rightarrow M_{k_{n+1}}$ are unital embeddings and assume that $\varphi_n(\cT_{k_n}) \subset \cT_{k_{n+1}}$. 
Then $\cT_\varphi = \overline{\bigcup_n \cT_{k_n}}$ is called a {\em triangular UHF (TUHF) algebra}.
\end{definition}
In contrast to Glimm's theorem we must take note of the embeddings as different embeddings lead to non-isomorphic algebras \cite{Power}. Hence, in the above definition $\varphi = \{\varphi_1,\varphi_2,\cdots\}$ is the collection of embeddings. Two of the simplest embeddings are:

\begin{definition}
The {\em standard} embedding of $\cT_k$ into $\cT_{k'}$ when $k | k'$ is 
\[
A \in \cT_k \ \mapsto \ I_{k'/k} \otimes A = \left[\begin{array}{cccc} A \\ & A \\ &&\ddots \\ &&&A\end{array}\right] \in \cT_{k'}
\]
\end{definition}

\begin{definition}
The {\em nest} or {\em refinement} embedding of $\cT_k$ into $\cT_{k'}$ when $k | k'$ is 
\[
A \in \cT_k \ \mapsto \ A\otimes I_{k'/k} \in \cT_k'
\]
or in other words
\[
\left[\begin{array}{cccc} a_{11} &\cdots&\cdots& a_{1k} \\ 0 & \ddots&&\vdots \\ \vdots &\ddots &\ddots&\vdots \\ 0 &\cdots&0&a_{kk}\end{array}\right]   \ \mapsto \ 
 \left[\begin{array}{cccc} 
a_{11}\cdot I_{k'/k} &\cdots&\cdots& a_{1k}\cdot I_{k'/k} \\ 0\cdot  I_{k'/k}  & \ddots&&\vdots \\ \vdots &\ddots &\ddots&\vdots \\ 0\cdot I_{k'/k} &\cdots&0\cdot I_{k'/k} &a_{kk}\cdot I_{k'/k}\end{array}\right].
\]
\end{definition}

Central to the theory of non-selfadjoint operator algebras is the notion of a C$^*$-envelope \cite{Arveson06, DritschelMcCullough, Hamana, Kakariadis12}, which can be thought of as the smallest C$^*$-algebra that contains the operator algebra. It is immediate in this case that the C$^*$-envelope, $C^*_e(\cT_\varphi)$, is equal to $C^*(\cT_\varphi) = \mathfrak A_\varphi$ because all UHF algebras are simple. 

Distinct from the theory of UHF algebras is that there is a partial order on ${\rm Proj}(\cT_\varphi)$ which is not the subprojection partial order.

\begin{definition}
If $p, q\in \cT$ are projections then we say $p \preceq q$ if there is a partial isometry $v\in \cT$ such that $vv^* = p$ and v*v = q.
\end{definition}

We will use $e_i^{k_n}$ to denote $e_{i,i}\in \cT_{k_n}$, the minimal projections at each level, and similarly $e_{i,j}^{k_n}$ to denote $e_{i,j}\in \cT_{k_n}$. From the previous definition we have $e_i^{k_n} \preceq e_j^{k_n}$ and $e_j^{k_n} \npreceq e_i^{k_n}$ for $i\leq j$.

%%%%%%%%%%%%%%

A subalgebra $\cT$ of a UHF algebra is {\em triangular} if $\cT \cap \cT^*$ is abelian. In the terminology of \cite{Power} our TUHF algebras are {\em strongly maximal triangular} in that there is no other triangular algebra sitting strictly between $\cT_\varphi$ and $\mathfrak A_\varphi$. 
Observe that $\varphi_{n}(\cT_{k_n} \cap \cT_{k_n}^*) \subset \cT_{k_{n+1}} \cap \cT_{k_{n+1}}^*$, that is the diagonal is mapped to the diagonal.
So there is a {\em maximal abelian self-adjoint subalgebra} (masa) $C_\varphi \subset \cT_\varphi$ defined as 
\[
C_\varphi \ = \ \cT_\varphi \cap \cT_\varphi^* \ =  \ \overline{\bigcup_n \cT_{k_n} \cap \cT_{k_n}^*} \ \simeq \  \overline{\bigcup_n C_{n}} \ := \ \overline{\bigcup_n \mathbb C^{k_n}}.
\]
Hence, $C_\varphi$ is an AF C$^*$-algebra and $C_\varphi \simeq C(X)$ where the Gelfand space is a generalized Cantor set:
\[
M(C_\varphi) = X = \prod_{n\geq 1} \left[\frac{k_n}{k_{n-1}}\right],
\]
with $k_0 = 1$ to make the formula work and where $[k] = \{0,1,\cdots, k-1\}$.
We will often refer to $C_\varphi$ as the diagonal of $\cT_\varphi$. 
%There is a unique trace $\tau$ on $\cT_\varphi$ given by setting $\tau(e_i^{(k_n)}) = \frac{1}{k_n}$ and linearly extending. 
%If trace is included show how it is constructed: each trace on a finite level works with the embeddings
For each point $x\in X$ there is a unique sequence of projections
\[
e_{i_1}^{k_1} \geq e_{i_2}^{k_2} \geq e_{i_3}^{k_3} \geq \cdots
\]
with $x(e_{i_n}^{k_n}) = 1$ for all $n\geq 1$. Define a partial order on $X$ by letting the following be equivalent for $x = (x_n)_{n\geq 1}, y=(y_n)_{n\geq 1} \in \prod_{n\geq 1} \left[\frac{k_n}{k_{n-1}}\right] = X$ which have sequences of projections $e_{i_n}^{k_n}$ and $e_{j_n}^{k_n}$ respectively:
\begin{itemize}
\item[1)] $x\leq y$,
\item[2)] $\exists\: n$ such that $(x_1,\cdots, x_n) \leq (y_1,\cdots, y_n)$ in the lexicographic order and $x_{n'} = y_{n'}, \forall n'>n$,
\item[3)] $\exists\: n$ such that $e_{i_n}^{k_n} \preceq e_{j_n}^{k_n}$ and $e_{i_{n'}}^{k_{n'}} = e_{i,j}^{k_n} e_{j_{n'}}^{k_{n'}} (e_{i,j}^{k_n})^*$ for all $n' > n$. 
\end{itemize}
Thus, this is a partial order on tail-equivalent sequences. Let $E_{ij}^{k_n}$ be all such pairs $(x,y)\in X\times X$ that depend on $i,j$ and $n$ in the above definition.

\begin{definition}
The {\em topological binary relation of $\cT_\varphi$ relative to $C_\varphi$} is 
\[
R(\cT_\varphi) = \bigcup\{ E_{ij}^{k_n} : e_{i,j}^{k_n} \in \cT_\varphi, n\geq 1\},
\]
equipped with the topology defined by basic clopen sets
\[
\{x \in X : x(e_i^{k_n}) = 1\}, \ \ n\geq 1, 1\leq i\leq k_n.
\]
\end{definition}

Lastly, we define the {\em normalizer} of $C_n$ in $\cT_{k_n}$ as
\[
N_{C_n}(\cT_{k_n})  \ = \ \{v\in \cT_{k_n} \ {\rm partial \ isometry} \ : vC_{n}v^* \subset C_n, v^*C_n v \subset C_n \}.
\]
It is not hard to see that any element of $N_{C_n}(\cT_{k_n})$ is the multiplication of a diagonal unitary by a partial permutation matrix, that is, where there is at most one 1 in each row and column.
We say that an embedding $\varphi: \cT_{k_n} \rightarrow \cT_{k_{n+1}}$ is a {\em regular embedding} if it takes partial permutation matrices to partial permutation matrices, which in turn implies that $\varphi(N_{C_n}(\cT_{k_n})) \subset N_{C_{n+1}}(\cT_{k_{n+1}})$. Note that the standard and nest embeddings are regular embeddings.

In the same way, define the normalizer of $C_\varphi$ in $\cT_\varphi$:
\[
N_{C_\varphi}(\cT_\varphi) \ = \ \{v\in \cT_{\varphi} \ {\rm partial \ isometry} \ : vC_\varphi v^* \subset C_\varphi, v^*C_\varphi v \subset C_\varphi \}.
\]
The following lemma by Power gives a decomposition of any element in the normalizer into a product of a unitary and a partial permuation matrix. Note that $U(C_\varphi)$ denotes the unitary group of $C_\varphi$.
\begin{lemma}[\cite{Power}, Lemma 5.5]\label{normalizer}
Let $\cT_\varphi$ have regular embeddings. Then $v\in N_{C_\varphi}(\cT_\varphi)$ if and only if $v=dw$ where $w\in N_{C_n}(\cT_{k_n})$, for some $n$, and $d \in U(C_\varphi)$, a diagonal unitary. Moreover, $w$ can be chosen to be a partial permutation matrix which makes the decomposition unique.
\end{lemma}

%%%%%%%%%%%%%%%%%%%%%%%%%%%%%%%%%%%%%%%%%%%%%%%
%%%%%%%%%%%%%%%%%%%%%%%%%%%%%%%%%%%%%%%%%%%%%%%
%%%%%%%%%%%%%%%%%%%%%%%%%%%%%%%%%%%%%%%%%%%%%%%

\section{Isometric automorphisms}

Let $\cT_\varphi$ be a TUHF algebra and $\Aut(\cT_\varphi)$ denote the isometric automorphism group. %[Say something about why it is meaningful to only talk about the isometric ones.] 
Such an automorphism will preserve the masa, the partial order on projections and the normalizer.

\begin{theorem}[ \cite{Power}, Theorem 7.5 ] \label{PowerThm}
Let $C_\varphi \subset \cT_\varphi \subset \mathfrak A_\varphi$ and $C_\psi \subset \cT_\psi \subset \mathfrak A_\psi$ be the algebras defined for two sequences of embeddings $\varphi$ and $\psi$. Then the following are equivalent:
\begin{enumerate}
\item There is an isometric isomorphism $\theta : \cT_\varphi \rightarrow \cT_\psi$ with $\theta(C_\varphi) = C_\psi$.
\item The topological binary relations $R(\cT_\varphi)$ and $R(\cT_\psi)$ are isomorphic as topological relations.
\item There is a $*$-isomorphism $\tilde\theta : \mathfrak A_\varphi \rightarrow \mathfrak A_\psi$ with $\tilde\theta(\cT_\varphi) = \cT_\psi$ and $\tilde\theta(C_\varphi) = C_\psi$.
\end{enumerate}
\end{theorem}

Furthermore, by  \cite[Corollary IV.5.8]{Dav} all automorphisms of $\mathfrak A_\varphi$ are approximately inner, i.e. the pointwise limit of inner automorphisms. Hence, by the previous theorem the automorphisms in $\Aut(\cT_\varphi)$ are just restrictions of approximately inner automorphisms. Consider now, that the only unitaries in $\cT_\varphi$ live in the masa, that is $U(\cT_\varphi) = U(C_\varphi)$. Since we refer to $C_\varphi$ as the diagonal of $\cT_\varphi$ this leads us to the following definition:

\begin{definition}
An approximately inner (or just inner) automorphism of $\cT_\varphi$ is called a {\em approximately diagonal} automorphism. We denote this group by $\overline{\Inn}(\cT_\varphi)$. More specifically, $\gamma\in \overline{\Inn}(\cT_\varphi)$ if there exists $U_n \in U(C_\varphi)$ such that 
\[
\lim_{n\rightarrow\infty} U_nAU_n^* = \gamma(A), \ \ \forall A\in \cT_\varphi.
\]
\end{definition}

Now because $U(C_\varphi)$ is commutative we immediately get that $\overline{\Inn}(\cT_\varphi)$ is commutative as well.

Define as well the outer automorphism group:
\[
\Out(\cT_\varphi) := \Aut(\cT_\varphi)/\overline{\Inn}(\cT_\varphi). 
\]

\begin{proposition}
$\cT_\varphi$ is always isometrically isomorphic to a triangular UHF algebra with regular embeddings.
\end{proposition}
\begin{proof}
For each $n\geq 1$ define a new function $\psi_n:\cT_{k_n} \rightarrow \cT_{k_{n+1}}$ by first defining $\psi_n(e_i^{k_n}) = \varphi_n(e_i^{k_n})$ and then defining $\psi_n(e_{i,j}^{k_n})$ in the best possible way. In particular, if $e_i^{k_n} \preceq e_j^{k_n}$ then 
\[
\psi_n(e_i^{k_n}) = \sum_{m=1}^{k'/k} e_{i_m}^{k_{n+1}} \preceq \psi_n(e_j^{k_n}) = \sum_{m=1}^{k'/k} e_{j_m}^{k_{n+1}} 
\]
with $i_m \leq i_{m+1}, j_m \leq j_{m+1}$ and $i_m \leq j_m$ and so define 
$\psi_n(e_{i,j}^{k_n}) = \sum_{m=1}^{k'/k} e_{i_m, j_m}^{k_{n+1}}$. Hence, $\psi_n$ is a regular embedding since it takes the partial permutation matrices of $\cT_{k_n}$ to partial permutations in $\cT_{k_{n+1}}$.

Thus, $C_\varphi = C_\psi$ with the even stronger condition that $R(\cT_\varphi) = R(\cT_\psi)$ since this is all determined by the partial order ``$\preceq$" which is unchanged using the $\psi$ embeddings.
Therefore, by Theorem \ref{PowerThm} $\cT_\varphi \simeq \cT_\psi$.
\end{proof}

\begin{lemma}\label{unclosed}
Let $\theta\in \Aut(\cT_\varphi)$. Then there exists $\gamma\in \overline{\Inn}(\cT_\varphi)$ such that 
\[
\gamma\circ\theta(\cup_{n\geq 1} \cT_{k_n}) = \cup_{n\geq 1} \cT_{k_n}.
\]
\end{lemma}
\begin{proof}
Assume without loss of generality that $\cT_\varphi$ has regular embeddings, since $\overline{\Inn}(\cT_\varphi)$ is isomorphism invariant. Let $n_1\geq 1$ be big enough such that $\theta(\proj(\cT_{k_1})) \subset \proj(\cT_{k_{n_1}})$ and using Lemma \ref{normalizer}, $\theta(e_{i,i+1}^{k_1}) = d_i w_i \in N_{C_\varphi}(\cT_\varphi)$ with $d_i \in U(C_\varphi)$ and $w_i\in N_{C_{n_1}}(\cT_{k_{n_1}})$, a partial permutation matrix, for $1\leq i < k_1$.

Set $u_1 = I \in C_\varphi$ and $u_2 \in U(C_\varphi)$ such that $u_2 = w_1^*d_1^*w_1$. Now, recursively define $u_i\in U(C_\varphi)$ by
\[
u_i = w_{i-1}^*d_{i-1}^*u_{i-1}w_{i-1}, \ \ {\rm for} \ \ 2 < i \leq k_1.
\]
Set $U_1 = \sum_{i=1}^{k_1} \theta(e_{i}^{k_1}) u_i \in U(C_\varphi)$ and notice that
\[
U_1^*\theta(e_{i,i+1}^{k_1})U_1  =  u_i^*\theta(e_{i,i+1}^{k_1})u_{i+1} = u_i^*(d_iw_i)u_{i+1} = w_i \in \cT_{k_{n_1}}.
\]
Thus, $U_1^*\theta(\cT_{k_1})U_1 \subset \cT_{k_{n_1}}$.

In the same way there exists $n_2\geq n_1$ and $U_2\in U(C_\varphi)$ such that $U_2^*\theta(\cT_{k_{n_1}})U_2 \subset \cT_{k_{n_2}}$.
Since the following are both regular embeddings they must be equal: 
\[
U_2^*\theta(\varphi_{k_{n_1}-1}\circ\cdots\circ\varphi_1(\cT_{k_1}))U_2 = \varphi_{k_{n_2}-1}\circ\cdots\circ\varphi_{k_{n_1}}(U_1^*\theta(\cT_{k_1})U_1).
\]
Repeating this we recursively get $n_{m+1} \geq n_m$ and $U_{m+1} \in U(C_\varphi)$ such that $U_{m+1}^*\theta(\cT_{k_{n_m}})U_{m+1} \subset \cT_{k_{n+1}}$ with $U_{m+1}U_m^*|_{\theta(\cT_{k_m})} = I$.

Therefore, the sequence $U_m$ defines an approximately inner automorphism $\gamma\in \overline{\Inn}(\cT_\varphi)$ and $\gamma\circ\theta(\cup_{n\geq1} \cT_{k_n}) = \cup_{n\geq 1} \cT_{k_n}$. Furthermore, for every $n\geq 1$, $\gamma\circ\theta|_{\cT_{k_n}}$ is a regular embedding into some $\cT_{k_{n'}}$. 
\end{proof}

\begin{proposition}\label{diagAuto}
Let $\theta\in \Aut(\cT_\varphi)$ and $\theta(p) = p,$ for all $p\in {\rm Proj}(\cT_\varphi)$. Then $\theta$ is an approximately diagonal automorphism.
\end{proposition}
\begin{proof}
Assume that $\cT_\varphi$ has regular embeddings. By the previous Lemma there exists $\gamma\in \overline{\Inn}(\cT_\varphi)$ such that $\tilde\theta := \gamma\circ\theta$ preserves the unclosed union and from the end of the proof we may assume that $\tilde\theta|_{\cT_{k_n}}$ is a regular embedding into $\cT_{k_{n'}}$.

Hence, for $1\leq i < j \leq k_n$, 
\[
\varphi_{n'-1}\circ\cdots\circ\varphi_{n}(e_{i,j}^{k_n}) = \sum_{l=1}^{k_{n'}/k_n} e_{i_l,j_l}^{k_{n'}}.
\]
and so 
\[
\sum_{l=1}^{k_{n'}/k_n} e_{i_l}\tilde\theta(e_{i_l, j_l}^{k_{n'}}) e_{j_l} = \tilde\theta( \sum_{l=1}^{k_{n'}/k_n} e_{i_l,j_l}^{k_{n'}} ) =  \tilde\theta(e_{i,j}^{k_n}) \in \cT_{k_{n'}}
\]
because $\tilde\theta(p) = p$ for all projections $p$. However, $\tilde\theta|_{\cT_{k_n}}$ is a regular embedding so there is no other option than to have $\tilde\theta(e_{i_l,j_l}^{k_{n'}}) = e_{i_l,j_l}^{k_{n'}}$ and so $\tilde\theta(e_{i,j}^{k_n}) = \varphi_{n'-1}\circ\cdots\circ\varphi_n(e_{i,j}^{k_n})$. 

Therefore, $\tilde\theta = \id$ and so $\theta = \gamma^{-1} \in \overline{\Inn}(\cT_\varphi)$.
\end{proof}

%%%%%%%%%%%%%%%%

\begin{theorem}\label{semidirect}
$\Aut(\cT_\varphi) \simeq \overline{\Inn}(\cT_\varphi) \rtimes \Out(\cT_\varphi)$.
\end{theorem}
\begin{proof}
Let $\cT_\varphi$ have regular embeddings. Lemma \ref{unclosed} and Proposition \ref{diagAuto} tell us that there is a unique representative to each coset of $\Out(\cT_\varphi)$, denote the collections of these as $\cO \subset \Aut(\cT_\varphi)$. Thus, if $\theta\in \cO$ then it acts as a regular embedding at each finite level. Composition of regular embeddings gives a regular embedding so it is immediate that if $\theta, \tilde\theta\in \cO$ then $\theta\circ\tilde\theta \in \cO$. Finally, $\theta^{-1}$ must send partial permutation matrices to partial permutation matrices because $\theta\in \cO$. But then $\theta^{-1}|_{\cT_{k_n}}$ must be a regular embedding and so $\theta^{-1}\in\cO$ as well. Therefore, $\cO$ is a group and is isomorphic to $\Out(\cT_\varphi)$.

Furthermore, for $\theta \in \cO$ and $\gamma\in \overline{\Inn}(\cT_\varphi)$ we have that for any $p\in \proj(\cT_\varphi)$ 
\[
\theta^{-1}\circ\gamma\circ\theta(p) = \theta^{-1}(\theta(p)) = p
\]
because approximately diagonal automorphisms preserve projections. By Proposition \ref{diagAuto} this implies that $\theta^{-1}\circ\gamma\circ\theta\in \overline{\Inn}(\cT_\varphi)$, which gives an action of $\Out(\cT_\varphi)$ on $\overline{\Inn}(\cT_\varphi)$. Therefore the result follows.
\end{proof}

A set of totally ordered projections $e_1 \preceq \cdots \preceq e_n \in \cT_n$ when embedded into $\cT_m$ becomes a partition $A_1 \dot\cup \cdots \dot\cup A_n$ where $|A_i| = |A_j| = m/n$ and $A_i \leq A_{i+1}$ in the sense that the $j$th smallest element of $A_i$ is smaller than the $j$th smallest element of $A_{i+1}$. We will call $A$ an {\em ordered partition}.

Suppose we have two such ordered partitions $A = \dot\cup A_i$ and $B = \dot\cup B_i$ then we say $A \preceq B$ if for some $1\leq j\leq m$,
$j'\in A_i$ if and only if $j'\in B_i$ for all $1\leq j' < j$ and $j\in A_i, j\in B_{i'}$ with $i < i'$. In other words, the element where they differ occurs in an earlier set. Hence, this is a total order on ordered partitions of the same set.

\begin{lemma}
Let $A = \dot\cup_{i=1}^n A_i$ and $B = \dot\cup_{i=1}^n B_i$ be ordered partitions of $\{1,\cdots, m\}$ and suppose that $\varphi : \cT_m \rightarrow \cT_{m'}$ is a unital embedding. If $A\preceq B$ then $\varphi(A) \preceq \varphi(B)$.
\end{lemma}
\begin{proof}
Let $j\in A_i, j\in B_{i'}, i < i'$ be the first element that differs in the two partitions.
Consider the first elementary projection of $\varphi(e_j)\in \cT_{m'}$, say $e_{j_1} \leq \varphi(e_j)$ then $j_1 \in \varphi(A_i)$ and $j_1\in \varphi(B_{i'})$. Now let $j' < j_1$. Then $e_{j'} \preceq e_{j_1}$ which implies that $e_{j'} \leq \varphi(e_{j''})$ with $j'' < j$ but then $j'' \in A_i$ if and only if $j''\in B_i$ and so $j' \in \varphi(A_i)$ if and only if $j'\in \varphi(B_i)$. Therefore, $\varphi(A) \preceq \varphi(B)$.
\end{proof}

Consider two embeddings $\varphi, \psi: \cT_k \rightarrow \cT_{k'}$. We say that $\varphi \preceq \psi$ if and only if $\varphi(\{1\}\cup\cdots\cup\{k\}) \preceq \psi(\{1\}\cup\cdots\cup\{k\})$. By the previous proposition if $\varphi':\cT_{k'}\rightarrow \cT_{k''}$ is another embedding then $\varphi \preceq \psi$ implies that $\varphi'\circ\varphi \preceq \varphi'\circ\psi$. Note that if $\varphi \preceq \psi$ and $\psi\preceq \varphi$ then they agree on projections and furthermore, that two such embeddings are always comparable in this way.

\begin{proposition}
$\Out(\cT_\varphi)$ is torsion free.
\end{proposition}
\begin{proof}
Let $\theta\in \Aut(\cT_\varphi)$ such that it preserves the unclosed union and $\theta^m = \id$ for some $m\geq 1$. For any choice of $n_1\geq 1$ there exist $n_{m+1} \geq \cdots \geq n_2 \geq n_1$ such that 
\[
\theta(\cT_{k_{n_i}}) \subset \cT_{k_{n_{i+1}}}, \ \ {\rm for} \ \ 1\leq i\leq m.
\]
For ease of notation let $k_i := k_{n_i}$, $\varphi_i := \varphi_{n_i}$ and $\theta_i := \theta|_{\cT_{k_i}}$. This gives us the following identities:
\[
\varphi_m\circ\cdots\circ\varphi_1 = \theta_m\circ\cdots \circ \theta_1 \ \ {\rm and} \ \ \theta_{i+1}\circ\varphi_i = \varphi_{i+1}\circ\theta_i.
\]
If $\varphi_1 \preceq \theta_1$ then by the previous lemma 
\begin{center}
\begin{tabular}{r c l}
$\varphi_m\circ\cdots\circ\varphi_1 \ $ & $\preceq$ & $\ \varphi_m\circ\cdots\circ\varphi_3\circ\varphi_2\circ\theta_1 \ $ \\
& $=$ & $\ \varphi_m\circ\cdots\circ\varphi_3\circ\theta_2\circ\varphi_1$ \\
& $\preceq$ & $ \  \varphi_m\circ\cdots\circ\varphi_3\circ\theta_2\circ\theta_1\ $ \\
& $=$ & $\ \varphi_m\circ\cdots\circ\varphi_4\circ\theta_3\circ\theta_2\circ\varphi_1$ \\
& $\preceq$ &  $\ \cdots$ \\
& $\preceq$ & $\ \varphi_m\circ\cdots\circ\varphi_i\circ\theta_{i-1}\circ\cdots\circ\theta_1\ $ \\
& $=$ & $\ \varphi_m\circ\cdots \circ \varphi_{i+1}\circ\theta_i\cdots\circ\theta_2\circ\varphi_1$ \\
& $\preceq$ & $\ \cdots$ \\
& $\preceq$ & $\ \varphi_m\circ\theta_{m-1}\circ\cdots\circ\theta_1 \ $ \\
& $=$ & $\ \theta_m\circ\cdots\circ\theta_2\circ\varphi_1$ \\
& $\preceq$ & $ \ \theta_m\circ\cdots\circ\theta_1 \ $ \\
& $=$ & $\ \varphi_m\circ\cdots\circ\varphi_1.$
\end{tabular} \\
\end{center}
Hence, all of the inequalities are equalities which makes the first line give us that $\varphi_1 = \theta_1$ on $\proj(\cT_{k_1})$. The same holds true if we assume $\theta_1 \preceq \varphi_1$ and thus, $\theta(p) = p$ for all projections $p\in \cT_\varphi$ and by Proposition \ref{diagAuto} $\theta\in \overline{\Inn}(\cT_\varphi)$.
Therefore, $\Out(\cT_\varphi)$ is torsion free.
\end{proof}

%%%%%%%%%%%%%%%%%%%%%%%%%%%%%%%%%%%%
%%%%%%%%%%%%%%%%%%%%%%%%%%%%%%%%%%%%
%%%%%%%%%%%%%%%%%%%%%%%%%%%%%%%%%%%%

\section{The alternating embedding}

\begin{definition}
We say that $\varphi$ is an {\em alternating embedding} if $k_n = s_nt_n, n\geq 1$ with $s_n | s_{n+1}$ and $t_n| t_{n+1}$ and 
\[
\varphi_n(A) = I_{s_{n+1}/s_n} \otimes A \otimes I_{t_{n+1}/t_n}.
\]
\end{definition}

This is called alternating because $\varphi_n$ is a standard embedding of size $s_{n+1}/s_n$ followed by a nest embedding of size $t_{n+1}/t_n$, though the order does not matter as tensoring is associative. To each such embedding associate a pair of supernatural numbers $(s_\varphi,t_\varphi)$ where $s_\varphi = \prod_{n\geq 1} \frac{s_{n+1}}{s_n}$ and $t_\varphi = \prod_{n\geq 1} \frac{t_{n+1}}{t_n}$, the supernatural numbers of the standard and nest embeddings treated separately.

For these algebras there is a version of Glimm's theorem, that an alternating TUHF is characterized by a pair of supernatural numbers up to finite rearranging:

\begin{proposition}[\cite{Power}, Theorem 9.6]\label{GlimmTUHF} 
Let $\cT_\varphi$ and $\cT_\psi$ have alternating embeddings. Then $\cT_\varphi$ is isometrically isomorphic to $\cT_\psi$ if and only if there exists $r\in \mathbb Q$ such that $s_\varphi = r\cdot s_\psi$ and $t_\varphi = r^{-1}\cdot t_\psi$.
\end{proposition}

\begin{proposition}
Let $\cT_\varphi$ have an alternating embedding. To every prime $p$ that infinitely divides both $s_\varphi$ and $t_\varphi$ there is a non-diagonal automorphism of $\cT_\varphi$, called a {\em shift automorphism} and denoted $\theta_p$.
\end{proposition}
\begin{proof}
Without loss of generality, by dropping to a subsequence of the $k_n$, we may assume that $p | \frac{s_{n+1}}{s_n}$ and $p | \frac{t_{n+1}}{t_n}$.
Define a map $\theta_p:\bigcup_{n\geq 1} \cT_{k_n} \rightarrow \bigcup_{n\geq 1} \cT_{k_n}$ by
\[
A\in \cT_{k_n} \mapsto \theta_p(A) = I_{\frac{p s_{n+1}}{s_n}} \otimes A \otimes I_{\frac{t_{n+1}}{p t_n}} \in \cT_{k_{n+1}}. 
\]
First off, $\theta_p$ is well-defined: 
\[
\theta_p(\varphi_n(A)) =  I_{\frac{p s_{n+2}}{s_{n+1}}} \otimes \left( I_{\frac{s_{n+1}}{s_n}} \otimes A \otimes I_{\frac{t_{n+1}}{t_n}} \right) \otimes I_{\frac{t_{n+2}}{p t_{n+1}}}
\]\[
= I_{\frac{s_{n+2}}{s_{n+1}}} \otimes \left( I_{\frac{ps_{n+1}}{s_n}} \otimes A \otimes I_{\frac{t_{n+1}}{pt_n}} \right) \otimes I_{\frac{t_{n+2}}{ t_{n+1}}} = \varphi_{n+1}(\theta_p(A)).
\]
Note that $\theta_p(e_1^{(k_n)}) \neq \varphi_{n}(e_1^{(k_n)})$ and so if this extends to an automorphism it will not be approximately diagonal.
Second, $\theta_p^{-1}$ is defined in the most obvious way:
\[
\theta_p^{-1}(\theta_p(A)) = I_{\frac{s_{n+2}}{ps_{n+1}}} \otimes \left( I_{\frac{ps_{n+1}}{s_n}} \otimes A \otimes I_{\frac{t_{n+1}}{pt_n}} \right) \otimes I_{\frac{pt_{n+2}}{ t_{n+1}}} 
\]\[
= I_{\frac{s_{n+2}}{s_n}} \otimes A \otimes I_{\frac{t_{n+2}}{t_n}} = \varphi_{n+1}(\varphi_n(A)).
\]
Similarly, $\theta_p(\theta_p^{-1}(A)) = \varphi_{n+1}(\varphi_n(A))$ as well. Hence, $\theta_p$ is an isometric automorphism on the unclosed union and so extends to be an isometric automorphism of $\cT_\varphi$.
\end{proof}

Let $p_1,\cdots, p_m$ be distinct primes that infinitely divide $s_\varphi$ and $t_\varphi$ and $\delta_1,\cdots, \delta_m \in \mathbb N$. For $u = \prod_{i=1}^m p_i^{\delta_i}$ define $\theta_u \in \Aut(\cT_\varphi)$ to be 
\[
\theta_u = \theta_{p_1}^{\delta_1}\circ\cdots\circ\theta_{p_m}^{\delta_m}.
\]
Note that the order of the $p_i$ does not matter as all of these automorphisms commute.

%%%%%%%%%%%%%%%%

We shift focus now back to ordered partitions. Before proving the main theorem of the section we first need two definitions and two technical lemmas.

Recall that $P = \dot\cup_{i=1}^n P_i$ is an ordered partition if $|P_1| = \cdots = |P_{n}|=m$ and $P_1 \leq P_2\leq \cdots \leq P_n$. This ordering can also be given by letting $P_i = \{p_{1,i},\cdots, p_{m,i}\}$ with $p_{1,i} < p_{2,i} < \cdots < p_{m,i}$ and then $P_i \leq P_j$ gives $p_{k,i} < p_{k,j}$ for every $1\leq k\leq m$.

We will call $P = \dot\cup_{i=1}^n P_i$ an {\em ordered subpartition} 
if $|P_1| \geq |P_2| \geq \cdots |P_n|$ and $P_i \leq P_j$ for $1\leq i < j\leq n$, meaning that $p_{l,i} < p_{l,j}$ for all $1\leq l\leq |P_j|$.

\begin{lemma}\label{Prestriction}
Let $P = \dot\cup_{i=1}^n P_i = \{1,\cdots, m\}$ be an ordered partition. Then for $1\leq m'\leq m$ we have that
\[
P\cap \{1,\cdots, m'\} = \dot\cup_{i=1}^n (P_i \cap \{1,\cdots, m'\} )
\]
is an ordered subpartition.
\end{lemma}
\begin{proof}
If $P_i \leq P_j$ then the $k$th smallest element of $P_i$ precedes the $k$th smallest element of $P_j$. Hence, if the latter is in $\{1,\cdots, m'\}$ then the former will be as well, and so, $P_i \cap \{1,\cdots, m'\} \leq P_j \cap \{1,\cdots, m'\}$. 
\end{proof}

A subset $R$ is called a {\em run} if whenever $i < j < k$ and $i,k\in R$ then $j\in R$. If $R$ and $S$ are runs we say that $R < S$ if $r < s$ for all $r\in R$ and $s \in S$.

\begin{lemma}\label{Psize}
Let $R_1 < R_2 < \cdots < R_n$ be runs in $\{1,\cdots, r\}$ and $S_1 < \cdots < S_n < S_{n+1}$ be runs in $\{1,\cdots, s\}$ with $|S_1| = \cdots = |S_n| \geq 1$. If $\theta$ is a unital embedding of $\cT_r$ into $\cT_s$ such that $\theta(R) = S$ as sets and $\theta(R_i) \supset S_i$ then $|R_1| \leq  \cdots \leq  |R_n|$.
\end{lemma}
\begin{proof}
Let $R_i = \{r_1^i,\cdots, r_{m_i}^i\}$ for $1\leq i\leq n$. Because $\theta$ is a unital embedding we know that it takes the indices 
\[
r_1^1 < r_2^1 < \cdots < r_{m_1}^1 < r_1^2 < r_2^2 < \cdots < r_{m_n}^n
\] 
to the ordered partition 
\[
\theta(r_1^1) \leq \theta(r_2^1) \leq \cdots \leq \theta(r_{m_1}^1) \leq \theta(r_{1}^2) \leq \cdots \leq \theta(r_{m_n}^n).
\]
In particular, they all have the same size, $|\theta(r_j^i)| = s/r$.
By the previous lemma this order is maintained when considering only the first part of $S$, leading to the ordered subpartition
\[
\theta(r_1^1)\cap (S_1 \cup \cdots \cup S_n) \leq \cdots \leq \theta(r_{m_n}^n)\cap (S_1\cup \cdots \cup S_n).
\]
Since $\theta(R_i)\supset S_i$ the ordered subpartition becomes
\[
\theta(r_1^1)\cap S_1 \leq \cdots \leq \theta(r_{n_1}^1)\cap S_1 \leq \theta(r_1^2)\cap S_2 \leq \cdots \leq \theta(r_{m_n}^n)\cap S_n.
\]
This implies that 
\[
|\theta(r_1^1)\cap S_1| \geq \cdots \geq |\theta(r_{n_1}^1)\cap S_1| \geq |\theta(r_1^2)\cap S_2| \geq \cdots \geq |\theta(r_{m_n}^n)\cap S_n|.
\]
However, if $i < i'$
\[
\sum_{k=1}^{m_i} |\theta(r_k^i)\cap S_i| = |S_i| = |S_{i'}| = \sum_{k=1}^{m_{i'}} |\theta(r_k^{i'})\cap S_{i'}|
\]
with every summand on the left being greater than every summand on the right, and so we must have $m_i \leq m_{i'}$.
In other words, 
\[
|R_1| \leq |R_2| \leq \cdots \leq |R_n|.
\]
\end{proof}

%%%%%%%Main Theorem%%%%%%%%%%%
%%%%%%%%%%%%%%%%%%%%%%%%%%%%
\begin{theorem}\label{MainThm}
Let $\cT_\varphi$ have an alternating embedding for $k_n = s_n t_n$ and $\theta\in \Aut(\cT_\varphi)$. Then there exists a approximately diagonal automorphism $\psi$ and $u,v\in \mathbb N$ such that $\theta = \theta_u\circ\theta_v^{-1}\circ\psi$. Moreover, this factorization is unique if $\gcd(u,v) = 1$.
\end{theorem}
\begin{proof}
Let $m\geq 1$ then there exist 
$m' \geq n \geq m$ such that
\[
\theta^{-1}(\proj(\cT_{k_m})) \subset \proj(\cT_{k_n}), \ \  {\rm and} \ \ \theta(\proj(\cT_{k_n})) \subset \proj(\cT_{k_{m'}}).
\]
We will use the language of ordered partitions. In particular, let 
\[
P = \dot\bigcup_{i=1}^{k_m} P_i = \varphi_{m'-1}\circ\cdots\circ\varphi_m(\{1\}\cup\cdots\cup \{k_m\}), 
\]
that is the image in $k_{m'}$ of the elementary projections in $k_m$. 
Writing these as the disjoint union of runs we get 
\[
P_i = \dot\bigcup_{j=1}^{s_{m'}/s_m} P_{j,i} \ \ {\rm and} \ \ P_{1,1} < P_{1,2} < \cdots < P_{1, k_m} < P_{2,1} \cdots < P_{s_{m'}/s_m, k_m}
\]
with $|P_{j,i}| = t_{m'}/t_m$, which is obvious from the alternating embedding.
Similarly, let 
\[
Q = \dot\bigcup_{i=1}^{k_m} Q_i = \theta^{-1}(\{1\}\cup\cdots\cup\{k_m\}), \ \ {\rm that\ is} \ \ \theta^{-1}(e_{i}^{k_m}) = \sum_{j\in Q_i} e_{j}^{k_n}.
\]
Also decompose this into runs
\[
Q_i = \dot\bigcup_{j=1}^{s} Q_{j,i} \ \ {\rm and} \ \ Q_{1,1} < Q_{1,2} < \cdots < Q_{s, k_m}
\] 
where many of the $Q_{j,i}$ may be empty, but there are never $k_m -1$ empty $Q_{j,i}$ all in a row because if this was not so then we could represent the partition as a shorter sequence. Note that $Q_{1,1}$ and $Q_{s, k_m}$ are nonempty.

\vskip 6 pt
\noindent {\bf Claim:} $|Q_{1,1}| = |Q_{1,2}| = \cdots = |Q_{1,k_m}|$. 
\vskip 6 pt

\noindent {\bf Proof of Claim:} \\
\noindent First, we know that
\[
P_{1,i} = P_i \cap P_{1,i} = \theta(\theta^{-1}(e_i^{k_m}))\cap P_{1,i} = \theta(Q_i)\cap P_{1,i} =   \bigcup_{j=1}^{k_n/k_m} \theta(Q_{j,i}) \cap P_{1, i}.
\]
By Lemma \ref{Prestriction} we get an ordered subpartition by intersecting with $P_{1,1}$,
\[
\left(\theta(Q_{1,1}) \leq \theta(Q_{1,2}) \cdots \leq \theta(Q_{s,k_m}) \right) \bigcap P_{1, 1} 
\]
\[
  =  \theta(Q_{1,1})\cap P_{1,1} \leq \emptyset \leq  \cdots \leq \emptyset \leq \theta(Q_{2,1})\cap P_{1,1} \leq \emptyset \leq \cdots  
  \]
  \[
 \cdots \leq \emptyset \leq \theta(Q_{3,1})\cap P_{1,1} \leq \emptyset \leq \cdots 
 \leq \emptyset \leq \theta(Q_{s,1})\cap P_{1,1} \leq \emptyset \leq \cdots \leq \emptyset
\]
which implies that if any $\theta(Q_{j,1})\cap P_{1,1}$ is nonempty then all the intermediate $Q_{1,1} < Q_{j',i'} <  Q_{j,1}$ must be empty to remain an ordered subpartition under the above restriction, but this contradicts the requirement that there cannot be $k_m -1$ empty $Q_{j',i'}$ in a row. Therefore, $\theta(Q_{1,1}) \cap P_{1,1} = P_{1,1}$.
\vskip 6 pt

\noindent Again
\[
\left(\theta(Q_{1,1}) \leq \theta(Q_{1,2}) \leq \cdots \leq \theta(Q_{s, k_m}) \right) \bigcap (P_{1, 1} \cup P_{1,2}) 
\]
\[
  =  \theta(Q_{1,1})\cap P_{1,1} \leq \theta(Q_{1,2})\cap P_{1,2} \leq \emptyset \leq  \cdots \leq \emptyset \leq \theta(Q_{2,2})\cap P_{1, 2} \leq \emptyset \leq \cdots 
\]
to get that $\theta(Q_{1,2}) \cap P_{1,2} = P_{1,2}$. Repeat this recursively to get that $\theta(Q_{1,i})\cap P_{1,i} = P_{1,i}$. Noting that all $|P_{1,i}| = |P_{1,i'}|$ we have satisfied the hypotheses of Lemma \ref{Psize}. Hence, $|Q_{1,1}| \leq \cdots \leq |Q_{1,k_m}|$.
The reverse direction is given by the fact that $Q_{1,1} < \cdots < Q_{1,k_m}$ is the first part of an ordered partition. Therefore, $|Q_{1,1}| = \cdots = |Q_{1,k_m}|$ and the claim has been verified.

\vskip 12 pt

This tells us that any isometric automorphism of an alternating embedding TUHF sends the elementary projections from a finite level to a partition with a specific starting pattern, that is, one iteration of equal runs. We apply this to the elementary projections of $\cT_{k_n}$ to get that there exist runs
\[
R_1 \leq R_2 \leq R_3 \leq \cdots \leq R_{k_n} 
\]
such that $|R_i| = |R_j| = r \geq 1$, $\cup R_i = \{1,\cdots, k\}$, $k\leq k_{m'}$ and $\theta(e_{i}^{k_n}) \supset R_i$. 
\vskip 6 pt

Let $Q_{j,i}' = \cup_{l \in Q_{j,i}} R_l$ giving us runs with $|Q_{j,i}'| = |Q_{j,i}| \cdot r$ and $\theta(Q_{j,i}) \supset Q_{j,i}'$.
Then the following partitions
\[
P \cap \{1,\cdots, k\} = \theta(\theta^{-1}(\{1,\cdots, k_m\})) \cap \{1,\cdots, k\} = \theta(Q)\cap \{1,\cdots, k\}
\]
must be equal. Which implies that 
\[
\cup P_{j,i} \cap \{1,\cdots, k\} = Q_{1,1}' < Q_{1,2}' < Q_{1,3}' < \cdots < Q_{j,i}' < \cdots < Q_{s,k_m}',
\]
where both are decompositions into runs.
Hence, $P_{j,i} = Q_{j,i}'$ which implies that $t = |Q_{j,i}| = |Q_{j,i}'|/r = |P_{j,i}|/r = \frac{t_{m'}}{t_m r}$, they are all the same size. Therefore, for $A\in \proj(\cT_{k_m})$
\[
\theta^{-1}|_{\cT_{k_m}}(A) = I_{s} \otimes A \otimes I_{t}.
\]
We have then, that $t \cdot s \cdot k_m = k_n$. 
Let $\frac{s}{s_n/s_m} = \frac{u}{v}$ where $u = \prod_{i=1}^l p_i^{\delta_i}$ and $v = \prod_{j=1}^k q_j^{\epsilon_j}$ with $p_1,\cdots, p_l, q_1,\cdots, q_k$ distinct primes and $\delta_1,\cdots, \delta_l, \epsilon_1,\cdots, \epsilon_k \in \mathbb N$. Because $st = \frac{k_n}{k_m} = \frac{s_n}{s_m}\frac{t_n}{t_m}$ then $\frac{t}{t_n/t_m} = \frac{v}{u}$. This gives us that $v | \frac{s_n}{s_m}$ and $u | \frac{t_n}{t_m}$.
Hence, for $A\in \proj(\cT_{k_m})$
\[
\theta^{-1}|_{\cT_{k_m}}(A) = I_{s} \otimes A \otimes I_{t} = I_{\frac{s_n}{s_m} \frac{u}{v}} \otimes A \otimes I_{\frac{t_n}{t_m}\frac{v}{u}}.
\]
\[
= \theta_{p_1}^{\delta_1}\circ\cdots\circ\theta_{p_l}^{\delta_l}\circ\theta_{q_1}^{-\epsilon_1}\circ\cdots\circ\theta_{q_k}^{-\epsilon_k}(A).
\]
Repeat this argument for any $\theta^{-1}(\proj(\cT_{k_{m'}})) \subset \proj(\cT_{k_{n'}})$, getting a similar result,
\[
\theta^{-1}|_{\cT_{k_{m'}}}(A) = \theta_{p'_1}^{\delta'_1}\circ\cdots\circ\theta_{p'_{l'}}^{\delta'_{l'}}\circ\theta_{q'_1}^{-\epsilon'_1}\circ\cdots\circ\theta_{q'_{k'}}^{-\epsilon'_{k'}}(A).
\]
However, these two descriptions must agree on $\cT_{k_m} \hookrightarrow \cT_{k_{m'}}$ and so $u=u', v=v'$ and note that $v | \frac{s_{n'}}{s_{m'}}$ and $u | \frac{t_{n'}}{t_{m'}}$. In this way we see that $\theta^{-1} = \theta_u\circ\theta_{v}^{-1}$ on the projections of $\cT_\varphi$ and that $v^\infty | s_\varphi, u^\infty | t_\varphi$. Finally then, by Proposition \ref{diagAuto} there exists a approximately diagonal automorphism $\psi$ such that $\theta = \theta_u^{-1}\circ\theta_v\circ\psi$ which also gives that $u^\infty | s_\varphi, v^\infty | t_\varphi$. Uniqueness of the factorization when $\gcd(u,v) = 1$ is obvious since we have seen that shift automorphisms and their inverses commute with other such automorphisms.
Therefore, the result is established.
\end{proof}

%%%%%%%%%%%%%%%%%%%%%%%%%%%%%%%%%%%%%%%%%
\begin{corollary}[cf. \cite{Power2}, Theorem 1]
Let $\cT_\varphi$ have an alternating embedding. Then $\Out(\cT_\varphi) \simeq \mathbb Z^d$ where $d$ is the number of common prime factors that infinitely divide both $s_\varphi$ and $t_\varphi$. 
\end{corollary}

%%%%%%%%%%%%%%%%%%%%%%%%%%%%%%%%%%%%%%%%%%
%%%%%%%%%%%%%%%%%%%%%%%%%%%%%%%%%%%%%%%%%%
%%%%%%%%%%%%%%%%%%%%%%%%%%%%%%%%%%%%%%%%%%

\section{Tensoring TUHF algebras}

The following section provides a technique to create new automorphism groups from old. 
To this end, suppose that $\cT_\varphi = \overline{\cup_{n=1}^\infty \cT_{k_n}}$ and $\cT_\psi = \overline{\cup_{n=1}^\infty \cT_{j_n}}$ are TUHF algebras. 

We can create a new TUHF algebra
\[
\cT_{\varphi\otimes \psi}= \overline{\cup_{n=1}^\infty \cT_{k_n j_n}} 
\]
with unital embeddings $\varphi_n\otimes\psi_n : \cT_{k_n j_n} \rightarrow \cT_{k_{n+1} j_{n+1}}$ defined by tensoring the old embeddings
\[
\varphi_n\otimes\psi_n(A) = \varphi_n\otimes\psi_n([A_{i,i'}]_{i,i'=1}^{k_n}) = (\varphi_n\otimes I_{j_{n+1}})([\psi_n(A_{i, i'})]_{i,i'=1}^{k_n}).
\]
Note that the $\psi_n$ are $*$-extendable to all of $M_{j_n}$, meaning that $\psi_n$ is the restriction of a unital C$^*$-embedding from $M_{j_n}$ into $M_{j_{n+1}}$, which is used when $i<i'$ in the block matrix.
Therefore, 
\[
\cT_{\varphi\otimes\psi} = \overline{\cup_{n=1}^\infty \cT_{k_n j_n}} \supsetneq \overline{\cup_{n=1}^\infty \cT_{k_n}\otimes \cT_{j_n}} = \cT_\varphi \otimes \cT_\psi.
\]
The new TUHF algebra is thus strictly bigger than the tensor product of the two previous algebras, but it inherits the automorphic structure of the two. It should be noted that this tensor operation is not commutative. That is, $\cT_{\varphi\otimes \psi}$ and $\cT_{\psi\otimes\varphi}$ need not be isomorphic.

This new embedding gives that $M(\cT_{\varphi\otimes\psi}) = M(\cT_\varphi)\times M(\cT_\psi)$ with the order $((x_1,x_2), (y_1,y_2)) \in R(\cT_{\varphi\otimes\psi})$ if and only if 
$(x_1,y_1) \in R(\cT_\varphi)$ and $(x_2,y_2)\in R(\cT_\psi)$ if $x_1 = y_1$.

In the following, $G^{\oplus \infty}$ refers to the infinite direct sum of a group $G$, a subgroup of the infinite direct product where elements are infinite tuples with all but a finite number of entries equal to the identity.

\begin{theorem}
Let $\cT_\varphi$ and $\cT_\psi$ be TUHF algebras then 
\[
 \Aut(\cT_\psi)^{\oplus \infty} \rtimes \Aut(\cT_\varphi) \ \ \subseteq \ \ \Aut(\cT_{\varphi\otimes\psi}).
\]
\end{theorem}
\begin{proof}

Clearly $\Aut(\cT_\varphi) \hookrightarrow \Aut(\cT_{\varphi\otimes\psi})$ since if $\theta$ is an order preserving homeomorphism of $M(\cT_\varphi)$ then $\theta\times \id$ is an order preserving homeomorphism of $M(\cT_{\varphi\otimes\psi}) = M(\cT_\varphi)\times M(\cT_\psi)$; and so by Theorem \ref{PowerThm} we get an induced automorphism on $\cT_{\varphi\otimes\psi}$. The same argument works for the embedding $\Aut(\cT_\psi) \hookrightarrow \Aut(\cT_{\varphi\otimes\psi})$ as well.

Moreover, we see that if $X\subset M(\cT_\varphi)$ is a clopen subset and $\theta$ is an order preserving homeomorphism of $M(\cT_\psi)$ then 
\[
\id_{X}\times \theta + \id_{X^C}\times \id
\]
is an also order preserving homeomorphism of $M(\cT_{\varphi\otimes\psi})$.
Since clopen subsets of $M(\cT_\varphi)$ are in bijective correspondence with the projections of $\cT_\varphi$ then for each $n\geq 1$ we see that 
\[
\id_{X_1}\times \theta_1 + \cdots + \id_{X_{k_n}}\times \theta_{k_n}
\]
is an order preserving homeomorphism where $X_j$ is the clopen subset associated with $e_{j}^{(k_n)} \in \cT_{k_n}$ and $\theta_j$ is an order preserving homeomorphism on $M(\cT_\psi)$. Thus, $\Aut(\cT_\psi)^{k_n} \hookrightarrow \Aut(\cT_{\varphi\otimes\psi})$.

Therefore, we have that $\underrightarrow\lim \Aut(\cT_\psi)^{\oplus k_n} \subset \Aut(\cT_{\varphi\otimes\psi})$ where the direct limit has the following injective homomorphisms:
$\tilde\varphi_n:\Aut(\cT_\psi)^{\oplus k_n} \rightarrow \Aut(\cT_\psi)^{\oplus k_{n+1}}$ where
\[ 
\tilde\varphi_n(\gamma_1,\cdots, \gamma_{k_n}) \ = \
(\gamma_{i_1},\gamma_{i_2},\cdots, \gamma_{i_{k_{n+1}}}),
\]
with $e_{j}^{(k_{n+1})} \leq \varphi_{n}(e_{i_j}^{(k_n)})$, for $1\leq j\leq k_{n+1}$.
Note that the direct limit $\underrightarrow\lim \Aut(\cT_\psi)^{\oplus k_n}$ is equal to the infinite direct sum $\Aut(\cT_\psi)^{\oplus \infty}$.
 
Finally, we need to describe the action of $\Aut(\cT_\varphi)$ on the direct limit.
Taking $\theta$ and $\gamma$ as order preserving homeomorphisms in $M(\cT_\psi)$ and $M(\cT_\varphi)$ respectively,  and $X$ clopen in $M(\cT_\varphi)$ we get that 
\[
(\gamma\times \id) \circ (\id_X \times \theta + \id_{X^C}\times \id)\circ (\gamma^{-1}\times \id) = \id_{\gamma(X)} \times \theta + \id_{\gamma(X)^C}\times \id.
\]
Therefore, $\Aut(\cT_\psi)^{\oplus\infty} \rtimes \Aut(\cT_\varphi) \subseteq \Aut(\cT_{\varphi\otimes\psi})$. 
\end{proof}

\begin{corollary}
$\Out(\cT_\psi)^{\oplus \infty}\rtimes \Out(\cT_\varphi) \subseteq \Out(\cT_{\varphi\otimes \psi})$
\end{corollary}
\begin{proof}
By Theorem \ref{semidirect} the outer automorphisms of both $\cT_\varphi$ and $\cT_\psi$ are well defined subgroups given by those automorphisms which are regular embeddings when restricted to a finite level. This property is clearly preserved in the proof of the last theorem and so the result follows.
\end{proof}

This implies that there are non-abelian outer automorphism groups. However, these groups may not be equal as in the following example:

\begin{example}
Let $\cT_\varphi$ be the standard embedding algebra for $2^\infty$ and $\cT_\psi$ be the nest embedding algebra for $2^\infty$.
Then $\cT_{\varphi\otimes\psi}$ is the alternating algebra for $2^\infty$. Hence, $\Out(\cT_{\varphi\otimes\psi}) = \mathbb Z \neq \{0\} = \Out(\cT_\psi)^{\oplus\infty} \rtimes \Out(\cT_\varphi)$.
\end{example}

%%%%%%%%%%%%%%%%%%%%%%%%%%%%%%%%%%%%%%%%%%%%%%%
%%%%%%%%%%%%%%%%%%%%%%%%%%%%%%%%%%%%%%%%%%%%%%%
%%%%%%%%%%%%%%%%%%%%%%%%%%%%%%%%%%%%%%%%%%%%%%%

\section{Dilation theory}

All the definitions in this last section come from the paper of Davidson and Katsoulis \cite{DavKat}.
An operator algebra $\cA$ is said to be {\em semi-Dirichlet} if $\cA^*\cA \subset \overline{\cA + \cA^*}$ when $\cA$ is considered as a subspace of its C$^*$-envelope. Moreover, a unital operator algebra $\cA$ is {\em Dirichlet} if $\cA + \cA^*$ is norm dense in its C$^*$-envelope, $C^*_e(\cA)$.

\begin{lemma}
Triangular UHF algebras are Dirichlet.
\end{lemma}
\begin{proof}
For a TUHF algebra $\cT_\varphi$ we have the much stronger condition that $\mathfrak A_\varphi = \cT_\varphi + \cT_\varphi^*$. Therefore, because the UHF algebra is simple we immediately get the desired result.
\end{proof}

A unital operator algebra $\cA$ is said to have the {\em Fuglede property} if for every faithful unital $*$-representation $\pi$ of $C^*_e(\cA)$ we have $\pi(\cA)' = \pi(C^*_e(\cA))'$.

\begin{lemma}
Triangular UHF algebras have the Fuglede property.
\end{lemma}
\begin{proof}
Suppose $\pi$ is a faithful unital $*$-representation of $C_e^*(\cT_\varphi) = \overline{\cup_{k_n} M_{k_n}}$.
Then $\pi(\cT_{k_n})' = \pi(M_{k_n})'$ and so $\pi(\cT_\varphi)' = \pi(C_e^*(\cT_\varphi))'$.
\end{proof}

An operator algebra $\cA$ has {\em isometric commutant lifting (ICLT)} if whenever there is a completely contractive representation $\rho:\cA \rightarrow B(\cH)$ commuting with a contraction $X$, there is a coextension $\sigma$ of $\rho$ and an isometric coextension $V$ of $X$ on a common Hilbert space $\cK$ so that $\sigma(\cA)$ and $V$ commute.

\begin{proposition}
Triangular UHF algebras have isometric commutant lifting.
\end{proposition}
\begin{proof}
Let $\rho$ be a contractive representation of $\cT_\varphi$ on $\cH$ commuting with a contraction $X$. Without loss of generality assume that $\rho$ is also unital. Now $\rho$ is completely contractive when restricted to any $\cT_{k_n}$ and thus on a dense set of $\cT_\varphi$. Hence, $\rho$ is a completely contractive representation. By Arveson's Extension Theorem and Stinespring's Dilation Theorem there is a $*$-homomorphism $\pi$ and an isometry $V: \cH \rightarrow \cK$ such that $\rho(a) = V^*\pi(a)V, \forall a\in \cT_\varphi$. This argument was given by Paulsen and Power in \cite{PaulsenPower} but can also be found in \cite{Dav3}.

For each $n\geq 1$ we know that $X$ commutes with $\rho(\cT_{k_n})$ and so by  \cite[Corollary 20.23]{Dav3} there is an operator $Y_n$ on $\cK$ commuting with $\pi|_{M_{k_n}}$ such that $\|Y_n\| = \|X\|$ and
\[
P(\cH)Y_n^m\pi(A)|\cH = X^m\rho(A), \ \ \forall m\geq 0, A\in \cT_{k_n}.
\]
Since all the $Y_n$ are bounded by $\|X\|\leq 1$ there is a subsequence converging in the weak operator topology to $Y\in B(\cK)$ which clearly commutes with $\pi$.
Now, dilate $Y$ to a lower triangular unitary $V$ on $\cK^{(\infty)}$ which commutes with $\pi^{(\infty)}$ because $\pi$ commutes with $Y^*$ as well. 
Thus, by restricting to the coextension part of the dilation we see that we have a coextension of $\rho$ which commutes with an isometric coextension of $X$. Therefore, $\cT_\varphi$ has property ICLT.
\end{proof}

Let $\rho$ be a representation of a unital operator algebra $\cA$. Then a coextension $\sigma$ of $\rho$ is called {\em fully extremal} if whenever $\pi$ is a dilation of $\sigma$ which is also a coextension of $\rho$ then $\pi$ is just a direct sum, $\pi = \sigma \oplus \sigma'$.

\begin{definition}
A unital operator algebra $\cA$ has the {\em Ando property} if whenever $\rho$ is a representation of $\cA$ and $X$ is a contraction commuting with $\rho(\cA)$, then there is a fully extremal coextension $\sigma$ of $\rho$ commuting with an isometric coextension of $X$. 
\end{definition}

\begin{theorem}
Triangular UHF algebras have the Ando property.
\end{theorem}
\begin{proof}
The following commutant lifting properties are all listed in \cite{DavKat} and will not be defined as they only are used as stepping stones in the proof below.

\cite[Corollary 7.4]{DavKat} gives that ICLT implies MCLT and  \cite[Corollary 5.18]{DavKat} gives that being Dirichlet and having MCLT implies CLT and CLT$^*$. Lastly, by  \cite[Corollary 9.12]{DavKat} having the Fuglede property, CLT and CLT$^*$ implies that triangular UHF algebras have the Ando property.
\end{proof}

If $\cA$ is an operator algebra and $\theta$ is an automorphism, the semicrossed product is the operator algebra 
\[
\cA \times_\theta \mathbb Z_+
\]
that encapsulates the dynamical system $(\cA, \theta)$. This first occurs in the work of Arveson \cite{Arveson} with a more modern treatment given by \cite{KakariadisKatsoulis}. In particular, this is the universal operator algebra generated by all covariant representations $(\rho, T)$ where $\rho$ is a completely contractive representation of $\cA$ and a contraction $T$ such that
\[
\rho(a)T = T\rho(\theta(a)), \ \ \forall a\in \cA.
\]
The following corollary says that the C$^*$-envelope of a semicrossed product of a TUHF algebra with an automorphism is in fact a full crossed product algebra.

\begin{corollary}
Let $\cT_\varphi$ be a TUHF algebra and $\theta\in\Aut(\cT_\varphi)$ then 
\[
C^*_e(\cT_\varphi \times_\theta \mathbb Z_+) = C^*_e(\cT_\varphi) \times_\theta \mathbb Z = \mathfrak A_\varphi \times_\theta \mathbb Z.
\]
\end{corollary}
\begin{proof}
By \cite[Theorem 12.3]{DavKat} if $\theta$ is an isometric automorphism of $\cT_\varphi$ then because TUHF algebras have the Ando property $C_e^*(\cT_\varphi \times_\theta \mathbb Z_+) = C_e^*(\cT_\varphi) \times_\theta \mathbb Z$. Lastly, recall that $C_e(\cT_\varphi) \simeq \mathfrak A_\varphi$.
\end{proof}

We end with the following example:

\begin{example}
Suppose $\cT_\varphi$ is a TUHF algebra with the $2^\infty$ alternating embedding and consider the shift automorphism $\theta_2$. Now $\cT_\varphi$ is a non-selfadjoint subalgebra of the CAR algebra, $M_{2^\infty} = \bigotimes_{-\infty}^\infty M_2$. In this form $\theta_2$ extends to the so called Bernoulli shift on the CAR algebra, taking a tensor in $\bigotimes_{-\infty}^\infty M_2$ and shifting it to the right.

Bratteli, Kishimoto, R{\o}rdam and St{\o}rmer show in \cite{BKRS} that 
\[
M_{2^\infty} \times_{\theta_2} \mathbb Z \ \simeq \ \underrightarrow\lim M_{4^n}\otimes C(\mathbb T),
\]
a limit circle algebra with embeddings being two copies of the twice-around embedding. Moreover, this AT algebra is isomorphic to $M_{2^\infty} \otimes \mathfrak B$ where $\mathfrak B = \underrightarrow\lim M_{2^n}\otimes C(\mathbb T)$ is the Bunce-Deddens algebra \cite{BunceDeddens}, thanks to Mikael R\o rdam for pointing this last isomorphism out. Among other things, this implies that the crossed product is a unital simple C$^*$-algebra which falls into Elliott's classification.

Therefore, by the above Corollary: 
\[
C^*_e(\cT_\varphi \times_{\theta_2} \mathbb Z_+) \simeq M_{2^\infty} \otimes \mathfrak B.
\]
This leads to the question of whether the semicrossed product is itself isomorphic to a ``nice'' subalgebra of $M_{2^\infty} \otimes \mathfrak B$, for instance a tensor of two non-selfadjoint operator algebras sitting in the CAR algebra and the Bunce-Deddens algebra.

\end{example}

%%%%%%%%%%%%%%%%%%%%%%%%%%%%%%%%%%%%%%%%%%%%%%%
%%%%%%%%%%%%%%%%%%%%%%%%%%%%%%%%%%%%%%%%%%%%%%%
%%%%%%%%%%%%%%%%%%%%%%%%%%%%%%%%%%%%%%%%%%%%%%%

% The bibliography
\bibliographystyle{amsplain}

\end{document}